\documentclass[a4paper,oneside,10pt]{amsart}

\usepackage{parskip}

\usepackage{a4wide}
\usepackage[T1]{fontenc}
\usepackage[ansinew]{inputenc}
\usepackage{lmodern} 
\usepackage{graphicx}
\usepackage{amsmath}
\usepackage{amsthm}
\usepackage{amsfonts}
\usepackage{amssymb}
\usepackage{setspace}
\usepackage{mathrsfs}
\usepackage[all]{xy}
\usepackage{enumerate}
\usepackage{xcolor}
\usepackage[normalem]{ulem}   

\theoremstyle{plain}

\newtheorem{theorem}{Theorem}[section]

\newtheorem{corollary}[theorem]{Corollary}
\newtheorem{proposition}[theorem]{Proposition}
 \newtheorem{lemma}[theorem]{Lemma}

\theoremstyle{definition}
\newtheorem{remark}[theorem]{Remark}

 \newtheorem{definition}[theorem]{Definition}

\newtheorem*{theorem*}{Theorem}
\newtheorem*{proposition*}{Proposition}
\newtheorem*{definition*}{Definition}
\numberwithin{equation}{section}
\newenvironment{abc}{\begin{enumerate}[{\rm (a)}]}{\end{enumerate}}
\newenvironment{num}{\begin{enumerate}[{\rm 1.}]}{\end{enumerate}}
\newenvironment{iiv}{\begin{enumerate}[{\rm (i)}]}{\end{enumerate}}

\makeatletter
\newcommand{\opnorm}{\@ifstar\@opnorms\@opnorm}
\newcommand{\@opnorms}[1]{%
  \left|\mkern-1.5mu\left|\mkern-1.5mu\left|
   #1
  \right|\mkern-1.5mu\right|\mkern-1.5mu\right|
}
\newcommand{\@opnorm}[2][]{%
  \mathopen{#1|\mkern-1.5mu#1|\mkern-1.5mu#1|}
  #2
  \mathclose{#1|\mkern-1.5mu#1|\mkern-1.5mu#1|}
}
\makeatother

\def\Ell{\mathrm{L}}
\newcommand\CB[1]{{\color{black} #1}}
\newcommand\RH[1]{{\color{black} #1}}

\begin{document}
\title{Extrapolation of operator-valued multiplication operators} 
\author{Christian Budde}
\address{North-West University, Department of Mathematics, Unit for BMI, Potchefstroom 2531, South Africa}
\email{christian.budde@nwu.ac.za}
\author{Retha Heymann}
\address{Stellenbosch University, Mathematics Division, Faculty of Science, Merriman Avenue, 7600 Stellenbosch, South Africa}
\email{rheymann@sun.ac.za}
\begin{abstract}
We discuss $\Ell^p$ fiber spaces which appear, e.g., as extrapolation spaces of unbounded multiplication operators which in turn are motivated, for instance, by non-autonomous evolution equations. 
\end{abstract}

\thanks{The first author was supported by the DAAD-TKA Project 308019 ``\emph{Coupled systems and innovative time integrators}''.}

\def\LLL{\mathscr{L}}
\def\RR{\mathbb{R}}
\def\CC{\mathbb{C}}
\def\NN{\mathbb{N}}
\def\QQ{\mathbb{Q}}
\def\M{\mathcal{M}}
\def\dom{\mathrm{D}}
\def\dd{\mathrm{d}}
\def\disp#1{\displaystyle{#1}}
\def\BC{\mathrm{C}_{\mathrm{b}}}
\def\BUC{\mathrm{C_{\mathrm{ub}}}}
\def\Cinf{\mathrm{C}_{\mathrm{0}}}
\def\ee{\mathrm{e}}
\def\esssup{\mathop{\mathrm{ess}\,\sup}}
\def\If{\mathrm{I}}
\def\T{\mathcal{T}}
\def\F{\mathcal{F}}
\def\B{\mathcal{B}}
\def\semis{\mathscr{P}}
\def\o{s}
\def\Id{\mathrm{I}}

\date{}
\maketitle

\section*{Introduction}
Nagel-Nickel-Romanelli studied extrapolation spaces for unbounded operators in their paper \cite{NagelIdent} in Quaest.\ Math.\ in 1996. This paper extends this work by considering another scenario where extrapolation spaces occur. 


It is motivated by non-autonomous Cauchy problems. 
Such a problem can be treated as an abstract Cauchy problem on a Banach space $X$. It takes the form
\begin{align}\label{eqn:nACP}\tag{nACP}
\begin{cases}
\dot{u}(t)=A(t)u(t),&\quad t,s\in\RR,\ t\geq s,\\
u(s)=x,
\end{cases}
\end{align}
where $(A(t),\dom(A(t)))_{t\in\RR}$ is a family of linear operators on $X$ (see \cite[Chapter VI, Section 9]{EngelNagel}). The solution of such a problem, if it exists, is given by a so-called evolution family $(U(t,s))_{t\geq s}$, see \cite[Chapter VI, Def.~9.2]{EN}. 
One of the problems of evolution families is that, in contrast to $C_0$-semigroups, they do not yield a differentiable solution, e.g., it is possible that the map $t\mapsto U(t,s)x$ is differentiable only for $x=0$. 
	However, if we have a solution by means of an evolution family, we obtain a differentiable structure by means of a strongly continuous semigroup $(T(t))_{t\geq0}$ on the Bochner space $\Ell^p(\RR,X)$ ($1\leq p<\infty$) by
\[
(T(t)f)(s):=U(s,s-t)f(s-t),\quad t\geq0,\ f\in\mathrm{L}^p(\RR,X),\ s\in\RR
\]
(see \cite[Chapter 9, Part b., especially Remark 9.12)]{EngelNagel}).
One important challenge is to determine the exact domain of the corresponding generator $(G,\dom(G))$, \CB{since it determines the space of differentiability for the semigroup, cf. \cite[Chapter II, Lemma~1.3]{EN}}. By \cite[Thm.~3.4.7]{ThomPhD} operator-valued multiplication operators now arise naturally. Indeed, the evolution family provides a solution to \eqref{eqn:nACP} if and only if there exists an invariant core $D\subseteq\mathrm{W}^{1,p}(\RR,X)\cap\dom(\mathcal{A})$ with 
\[Gf=\mathcal{A}f-f'\]
for $f\in D$. Here $\mathcal{A}$ is the multiplication operator on $\Ell^p(\RR,X)$ defined by
$$
  \left( \mathcal{A}f
  \right)(\cdot) = A(\cdot)f(\cdot) \ \text{ for } \ f\in \Ell^p(\RR,X) \text{ with } A(\cdot)f(\cdot)\in\Ell^p(\RR,X),
$$ 
where $A(t), t\in\RR,$ is the family of linear operators in \eqref{eqn:nACP}.
One answer, in the special case where $A(t)\equiv A$ for some semigroup generator $(A,\dom(A))$, is due to Nagel, Nickel and Romanelli \cite[Sect.~4]{NagelIdent}. In \cite{Graser1997} T.~Graser studied bounded and unbounded operator-valued multiplication operators on the space of continuous functions $\mathrm{C}_0(\RR,X)$ as well as their extrapolation spaces. We will see that extrapolation spaces of multiplication operators on $\Ell^p(\RR,X)$ can be constructed similarly. Later on, S.~Thomaschewski studied properties of such multiplication operators on Bochner $\Ell^p$-spaces \cite[Sect.~2.2 \&~2.3]{ThomPhD} in connection with non-autonomous problems. \RH{In particular}, she connects multiplication semigroups with unbounded operator-valued multiplication operators. 
In order to construct our extrapolation spaces, the notion of fiber integrable functions is essential and leads to so-called $L^p$ fibre spaces, see \cite{HeymannPHD}.

We start this paper with some preliminaries on fiber integrable functions and continue with unbounded multiplication operators in the second section. In Section \ref{sec:ExtraFiberMultSemi} we discuss multiplication semigroups whose generators are multiplication operators. Furthermore, we determine the extrapolation spaces of such multiplication operators by means of $\Ell^p$-fiber spaces. We remark that the results of this paper appear in similar fashion in \cite[Chapter 3]{BuddePhD}.

\section{$\mathrm{L}^p$-fiber spaces}

Firstly, we recall the notion of a measurable Banach fiber bundle as it was introduced by R.~Heymann, cf. \cite[Def.~VI.1.i]{HeymannPHD}. To do so, let $(\Omega,\Sigma,\mu)$ be a $\sigma$-finite measure space. Furthermore, let $V$ be a complex vector space together with a family of seminorms $\left\{\opnorm{\cdot}_s:\ s\in\Omega\right\}$ on $V$. Assume that there exists a countable set of elements $\B:=\left\{b_k:\ k\in\NN\right\}\subseteq V$ such that $\B$ is a vector space over $\QQ+i\QQ$ and such that for each $k\in\NN$ the map $s\mapsto\opnorm{b_k}_s$ is measurable as a map from $\Omega$ to $\RR$. For every $s\in\Omega$ we define the set $N_s:=\left\{b_k\in\B:\ \opnorm{b_k}_s=0\right\}$ and take the completion of the quotient space $\B/N_s$ with respect to the induced norms $\left\|\cdot\right\|_s$ on $\B/N_s$. This Banach space is denoted by $X_s$.

\begin{definition}\label{def:mBfs}
The family of Banach spaces $(X_s,\left\|\cdot\right\|_s)_{s\in\Omega}$ is called a \emph{measurable Banach fiber bundle}.
\end{definition}

Next, we follow \cite[Def.~VI.1.iii]{HeymannPHD} in order to define what it means for a function $f:\Omega\to\bigcup_{s\in\Omega}{X_s}$ to be measurable.

\begin{definition}\label{def:fibrem}
Let $(X_s,\left\|\cdot\right\|_s)_{s\in\Omega}$ be a measurable Banach fiber bundle. We define a function $f:\Omega\rightarrow\bigcup_{s\in\Omega}{X_s}$ with $f(s)\in X_s$ for $\mu$-almost every $s\in\Omega$ to be \emph{fiber measurable} if it is almost everywhere a pointwise limit with respect to $\left\|\cdot\right\|_s$ of measurable simple functions with values in $\B$. More precisely, this means that there exists a sequence $(f_j)_{j\in\NN}$ of functions $f_j:\Omega\rightarrow\bigcup_{s\in\Omega}{X_s}$ such that
\begin{num}
	\item $\displaystyle{f_j=\sum_{i=1}^{n_j}{\left(b_{\pi(i)}+N_s\right)\textbf{1}_{\Omega_i}}}$, where $\pi:\NN\rightarrow\NN$, $n_j\in\NN$, $\Omega_i\in\Sigma$, $\Omega_i\cap\Omega_j=\varnothing$, $i\neq j$, and $b_i\in\B$ for each $1\leq i\leq n_j$,
	\item $\displaystyle{f(s)=\lim_{j\to\infty}{f_j(s)}}$ with respect to $\left\|\cdot\right\|_s$ for $\mu$-almost every $s\in\Omega$. 
\end{num}
The set of fiber measurable functions on $\Omega$ together with the pointwise addition and scalar multiplication is a $\mathbb{C}$-vector space, which we will call a \emph{measurable Banach fiber bundle}.
\end{definition}

Having the concept of measurability we continue with the notion of integrability for functions from $\Omega$ to measurable Banach fiber bundles, see \cite[Def.~VI.1.vi]{HeymannPHD}. Especially, we define what it means to be $p$-integrable for $1\leq p<\infty$. To do so, we remark that, for a fiber-measurable function $f$, the map $s\mapsto\left\|f(s)\right\|_s^p$ is measurable, since, by construction, the map $s\mapsto\left\|f(s)\right\|_s$ is a pointwise limit of measurable functions.

\begin{definition}\label{def:FibInt}
Let $1\leq p<\infty$. We call a fiber measurable function $f:\Omega\rightarrow\bigcup_{s\in\Omega}{X_s}$ \emph{fiber $p$-integrable} if the integral
\[
\int_{\Omega}{\left\|f(s)\right\|_s^p}\ \dd\mu(s),
\]
exists and is finite. In this case, we call
\[
\left\|f\right\|_p:=\left(\int_{\Omega}{\left\|f(s)\right\|_s^p}\ \dd\mu(s)\right)^{\frac{1}{p}}
\]
the \emph{$\Ell^p$-fiber norm} of $f$.
\end{definition}
 
\begin{remark}
\begin{iiv}
	\item Observe that the set of fiber $p$-integrable functions with pointwise addition and scalar multiplication is a vector space.
	\item The relation defined by $f\sim g$ $\Leftrightarrow$ $f=g$ $\mu$-almost everywhere is an equivalence relation on the set of fiber $p$-integrable functions.
	\item The set of equivalence classes of fiber $p$-integrable functions with the canonical vector space structure is called a \emph{$\mathrm{L}^p$-fiber space} and is denoted by $\Ell^p(\Omega,(X_s)_{s\in\Omega})$.
	\item By \cite[Prop.~VI.1.xi]{HeymannPHD} the space $\Ell^p(\Omega,(X_s)_{s\in\Omega})$ is a Banach space with respect to the $\Ell^p$-fiber norm.
	\item \CB{The notion of $\Ell^p$-fiber spaces is closely related to the one of measurable Banach bundles, cf. \cite{Gut1993} and \cite{GM2013}.}
\end{iiv}
\end{remark}

\section{Operator-valued multiplication operators}\label{sec:UnbMult}

The main objects of this section are 
operator-valued multiplication operators, cf. \cite[Def.~2.3.1]{ThomPhD}.

\begin{definition}\label{def:UnOpVMult}
Let $X$ be a Banach space and let $(M(s),\dom(M(s)))_{s\in\Omega}$ be a family of (possibly) unbounded linear operators on $X$, i.e., $M(s):\dom(M(s))\subseteq X\rightarrow X$ for $s\in\Omega$. The operator $(\M,\dom(\M))$ on $\Ell^p(\Omega,X)$ defined by
\begin{align*}
\dom(\M)&:=\left\{f\in\Ell^p(\Omega,X):\ f(s)\in\dom(M(s))\ \mu\text{-a.e.}, (s\mapsto M(s)f(s))\in\Ell^p(\Omega,X)\right\},\\
(\M f)(s)&:=M(s)f(s),\quad f\in\dom(\M), s\in\Omega,\ \mu\text{-almost everywhere},
\end{align*}
is called the corresponding \emph{
	operator-valued multiplication operator}. The operators $(M(s),\dom(M(s)))$, $s\in\Omega$, are called \emph{fiber operators}.
\end{definition}

As already mentioned in the beginning, the concept of unbounded multiplication operators was studied by S.~Thomaschewski \cite{ThomPhD} on Bochner $\Ell^p$-spaces in connection with non-autonomous Cauchy problems. Here we summarize some important results.

Firstly, the closedness of the fiber operators implies the closedness of the multiplication operator, see \cite[Lemma~2.3.4]{ThomPhD}.

\begin{lemma}
If $(M(s),\dom(M(s)))$ is closed for $\mu$-almost every $s\in\Omega$, then $(\M,\dom(\M))$ is closed.
\end{lemma}

In what follows we assume that $(\M,\dom(\M))$ is a closed operator-valued multiplication operator with closed fiber operators $(M(s),\dom(M(s)))_{s\in\Omega}$. The following result \cite[Lemma~2.3.5]{ThomPhD} shows that the resolvent operator of $(\M,\dom(\M))$ \RH{also gives} rise to a multiplication operator. For this we remind the reader of the following definition \cite[Def.~2.2.3]{ThomPhD}:
\[
\mathrm{L}^{\infty}\left(\Omega,\LLL_\mathrm{s}(X)\right):=\left\{M:\Omega\to\LLL(X):\ s\mapsto M(s)x\in\Ell^{\infty}(\Omega,X)\ \text{for all}\ x\in X\right\}.
\]

\begin{lemma}\label{lem:ResMultOp}
Let $(\M,\dom(\M))$ be a multiplication operator and assume that $\lambda\in\rho(\M)$ \RH{where $\rho(\M)$ denotes the resolvent set of $\M$}. Then $R(\lambda,\M)$ is a bounded multiplication operator, i.e., there exists $M\in\mathrm{L}^{\infty}\left(\Omega,\LLL_\mathrm{s}(X)\right)$ such that $(R(\lambda,\M)f)(s)=M(s)f(s)$ for all $f\in\Ell^p(\Omega,X)$.
\end{lemma}

Unfortunately, since the proof of the above lemma uses a characterization of bounded multiplication operators, the converse cannot be proved in a similar way and has not been proved in general, cf. \cite[Thm.~2.2.17]{ThomPhD}.
However, the following result \cite[Thm.~2.3.6]{ThomPhD} holds.

\begin{theorem}
Let $(\M,\dom(\M))$ be a densely defined closed von $\Ell^p(\Omega,X)$. Assume that there exists an unbounded sequence $(\lambda_n)_{n\in\NN}$ in $\rho(\M)$ such that for all $f\in\Ell^p(\Omega,X)$ one has\\ $\lim_{n\to\infty}{\lambda_nR(\lambda_n,\M)f}=f$. If $R(\lambda_n,\M)$ is a bounded multiplication operator for every $n\in\NN$, then there exists a family $(M(s),\dom(M(s)))_{s\in\Omega}$ of densely defined closed operators on $X$ such that $(\M,\dom(\M))$ is a multiplication operator with fiber operators $(M(s),\dom(M(s)))_{s\in\Omega}$. Furthere there exists a $\mu$-null-set $\mathscr{N}$ such that for every $s\in\Omega\setminus\mathscr{N}$ and for each $n\in\NN$ one has $\lambda_n\in\rho(M(s))$.
\end{theorem}

Finally, if $(\M,\dom(\M))$ is already supposed to be a multiplication operator on $\Ell^p(\Omega,X)$, then the resolvent of $\M$ and the resolvents of the fiber operators are related by the following result \cite[Prop.~2.3.7]{ThomPhD}.

\begin{proposition}\label{prop:ResolventMultFiber}
Let $(\M,\dom(\M))$ be a closed multiplication operator with closed fiber operators $(M(s),\dom(M(s)))_{s\in\Omega}$. 
\begin{abc}
	\item  If $\lambda\in\rho(M(s))$ for $\mu$-almost every $s\in\Omega$ and $R(\lambda,M(\cdot))\in\Ell^\infty\left(\Omega,\LLL_\mathrm{s}(X)\right)$, then $\lambda\in\rho(\M)$ and $(R(\lambda,\M)f)(s)=R(\lambda,M(s))f(s)$ for all $f\in\Ell^p(\Omega,X)$ and $\mu$-almost every $s\in\Omega$.
	\item If there exists an unbounded sequence $(\lambda_n)_{n\in\NN}$ in $\rho(\M)$ such that for all $f\in\Ell^p(\Omega,X)$ one has $\lambda_nR(\lambda_n,\M)f\to f$ for $n\to\infty$, then for $\mu$-almost all $s\in\Omega$ and all $n\in\NN$ one has $\lambda_n\in\rho(M(s))$ and $(R(\lambda_n,\M)f)(s)=R(\lambda_n,M(s))f(s)$ for all $f\in\Ell^p(\Omega,X)$ and $\mu$-almost every $s\in\Omega$.
\end{abc}
\end{proposition}

In \cite[Sect.~2.2.3]{ThomPhD} S.~Thomaschewski 
studies multiplication semigroups on $\Ell^p(\Omega,X)$.
The definition is as follows.

\begin{definition}
A $C_0$-semigroup $(\mathcal{T}(t))_{t\geq0}$ on $\Ell^p(\Omega,X)$ is called a \emph{multiplication semigroup} if for every $t\geq0$ the operator $\mathcal{T}(t)$ is a bounded multiplication operator, i.e., for every $t\geq0$ there exists $T_{(\cdot)}(t)\in\Ell^{\infty}\left(\Omega,\LLL_\mathrm{s}(X)\right)$ such that $(\mathcal{T}(t)f)(s)=T_s(t)f(s)$ for $\mu$-almost every $s\in\Omega$.
\end{definition}

By \cite[Thm.~2.3.15]{ThomPhD}, stated below, these multiplication semigroups have unbounded multiplication operators as generators.


\begin{theorem}\label{thm:MultGenC0Thom}
Let $(\M,\dom(\M))$ be the generator of a strongly continuous semigroup $(\mathcal{T}(t))_{t\geq0}$ on $\Ell^p(\Omega,X)$ such that $\left\|\mathcal{T}(t)\right\|\leq M\ee^{\omega t}$ for some $M\geq0$, $\omega\in\RR$ and for all $t\geq0$. The following are equivalent.
\begin{abc}
	\item $(\mathcal{T}(t))_{t\geq0}$ is a multiplication semigroup.
	\item $(\M,\dom(\M))$ is an unbounded operator-valued multiplication operator with fiber operators\\ $(M(s),\dom(M(s)))_{s\in\Omega}$. Moreover, for $\mu$-almost every $s\in\Omega$, $\lambda\in\rho(M(s))$ whenever $\mathrm{Re}(\lambda)>\omega$, $(R(\lambda,\M)f)(\cdot)=R(\lambda,M(\cdot))f(\cdot)$ and $(M(s),\dom(M(s)))$ is the generator of a $C_0$-semigroup $(T_s(t))_{t\geq0}$ such that $(\mathcal{T}(t)f)(s)=T_s(t)f(s)$ for all $t\geq0$.
\end{abc}
\end{theorem}


\section{Extrapolation of unbounded multiplication operators}\label{sec:ExtraFiberMultSemi}

We recall the construction of extrapolation spaces of unbounded operators, 
\CB{
as described, for example, in \cite[Chapter II, Sect.~5(a)]{EN}, \cite{N1997} or in a more general framework in \cite{BF2019}: Let $(T(t))_{t\geq0}$ be a $C_0$-semigroup on a Banach space $X$ with generator $(A,\dom(A))$. Without loss of generality one may assume that $0\in\rho(A)$. On $X$ one defines a new norm $\left\|\cdot\right\|_{-1}$ by
\[
\left\|x\right\|_{-1}:=\left\|A^{-1}x\right\|,\quad x\in X.
\]
The completion of $X$ with respect to $\left\|\cdot\right\|_{-1}$ is called the \emph{(first) extrapolation space} and will be denoted by $X_{-1}$. The original space $X$ is densely embedded in $X_{-1}$. By continuity one can extend the original semigroup $(T(t))_{t\geq0}$ to a $C_0$-semigroup $(T_{-1}(t))_{t\geq0}$ on $X_{-1}$. The corresponding generator is denoted by $(A_{-1},\dom(A_{-1}))$. As a matter of fact, one obtains $\dom(A_{-1})=X$ and $A_{-1}:X\to X_{-1}$ becomes an isometric isomorphism.}\\
We now consider a multiplication operator $(\M,\dom(\M))$ generating a multiplication semigroup $(\mathcal{T}(t))_{t\geq0}$, cf. Theorem \ref{thm:MultGenC0Thom}. By $(M(s),\dom(M(s)))_{s\in\Omega}$ we denote the fiber operators corresponding to $(\M,\dom(\M))$, i.e., $(\M f)(s)=M(s)f(s)$, $f\in\dom(\M)$ and $s\in\Omega$. By Theorem \ref{thm:MultGenC0Thom} the operator $(M(s),\dom(M(s)))$, $s\in\Omega$, generates a $C_0$-semigroup $(T_s(t))_{t\geq0}$. We assume without loss of generality that $0\in\rho(M(s))$ for $\mu$-almost every $s\in\Omega$. The extrapolated operators will be denoted by $(M_{-1}(s),\dom(M_{-1}(s)))$, $s\in\Omega$. The extrapolation space of $\Ell^p\left(\Omega,X\right)$ corresponding to the operator $(\M,\dom(\M))$ will be denoted by $\mathcal{X}:=\left(\Ell^p(\Omega,X)\right)_{-1}(\M)$. \CB{By the extrapolation procedure described above, $\mathcal{X}$} is formally given by
\[
\mathcal{X}=\left\{f\in\prod_{s\in\Omega}{X_{-1,s}}:\ \exists g\in\Ell^p(\Omega,X):\ f=M_{-1}(\cdot)g\right\}
\]
The norm $\left\|f\right\|_{\mathcal{X}}:=\left(\int_\Omega{\left\|f(s)\right\|^p_{-1,s}\ \dd\mu(s)}\right)^{1/p}$ turns $\mathcal{X}$ into a Banach space. 

\begin{lemma}\label{lem:ExtMultSemiAssumpGen}
Let $(\M,\dom(\M))$ be a multiplication operator on $\Ell^p(\Omega,X)$ with fiber operators\\ $(M(s),\dom(M(s)))_{s\in\Omega}$. Moreover, let $(\mathcal{T}(t))_{t\geq0}$ the multiplication semigroup \CB{of type $(M,\omega)$} generated by $(\M,\dom(\M))$. Moreover, denote by $(T_s(t))_{t\geq0}$ the $C_0$-semigroups of type $(M,\omega)$ generated by the fiber operators $(M(s),\dom(M(s)))_{s\in\Omega}$. The associated extrapolated semigroups are denoted by $(T_{-1,s}(t))_{t\geq0}$, $s\in\Omega$. Define
\[
(\mathcal{S}(t)f)(s):=T_{-1,s}(t)f(s),\quad t\geq0,\ f\in\mathcal{X},\ s\in\Omega.
\]
This defines a $C_0$-semigroup on $\mathcal{X}$ which is generated by the operator $(\M_{-1},\dom(\M_{-1}))$ defined by
\begin{align}\label{eqn:ExtraMultFiberOp}
(\M_{-1}f)(s)=M_{-1}(s)f(s),\quad \dom(\M_{-1})=\Ell^p(\Omega,X).
\end{align}
\end{lemma}

\begin{proof}
First of all, to see that $(\mathcal{S}(t))_{t\geq0}$ is indeed a semigroup is easy, since $(T_{-1,s}(t))_{t\geq0}$ is a semigroup for each $s\in\Omega$. As a matter of fact, $(T_{-1,s}(t))_{t\geq}$ is an extension of $(T_s(t))_{t\geq0}$ for each $s\in\Omega$ and hence $(\mathcal{S}(t))_{t\geq0}$ extends $(\mathcal{T}(t))_{t\geq0}$. Since, by construction, $\Ell^p(\Omega,X)$ is dense in $\mathcal{X}$, the semigroup $(\mathcal{S}(t))_{t\geq0}$ is strongly continuous. In order to show that the generator of $(\mathcal{S}(t))_{t\geq0}$ is of the form mentioned in the lemma, let us denote the generator of $(\mathcal{S}(t))_{t\geq0}$ by $(\mathcal{A},\dom(\mathcal{A}))$. \CB{We will show that $(\mathcal{A},\dom(\mathcal{A}))=(\M_{-1},\dom(\M_{-1}))$. Firstly, assume that $f\in\dom(\mathcal{A})$, then since $\lambda\in\rho(M_{-1}(s))$ for all $\lambda>\omega$ and almost every $s\in\Omega$ we conclude that $\lambda\in\rho(\M_{-1})$. By assumption $(\mathcal{A},\dom(\mathcal{A}))$ is the generator of $(\mathcal{S}(t))_{t\geq0}$ with $\left\|\mathcal{S}(t)\right\|\leq M\ee^{\omega t}$ meaning that $\lambda\in\rho(\mathcal{A})$. From surjectivity one obtain $f\in\dom(\mathcal{A})$, $g\in\mathcal{X}$ and $h\in\dom(\M_{-1})$ satisfying the following equality
\[
f=R(\lambda,\mathcal{A})g=R(\lambda,\mathcal{A})(\lambda-\M)h=R(\lambda,\mathcal{A})(\lambda-\mathcal{A})h=h\in\dom(\M_{-1}),
\]
showing that $\mathcal{A}\subseteq\M_{-1}$. For the converse, let $f\in\dom(\M_{-1})$ and observe that from the fact $f(s)\in\dom(M_{-1}(s))$ for almost every $s\in\Omega$ one conclude that
\[
(\mathcal{S}(t)f)(s)-f(s)=T_{-1,s}(t)f(s)-f(s)=\int_0^t{T_{-1,s}(t)M_{-1}(s)f(s)\ \dd\mu(s)}=\int_0^t{(\mathcal{S}(t)\mathcal{M}_{-1}f)(s)\ \dd\mu(s)},
\]
showing that $\M_{-1}\subseteq\mathcal{A}$, which concludes the proof.
}\end{proof}


The previous result shows actually that the extrapolated multiplication operator is again a multiplication operator. In this case the fiber operators are the extrapolated fiber operators\\ $(M_{-1}(s),\dom(M_{-1}(s)))_{s\in\Omega}$, i.e., \eqref{eqn:ExtraMultFiberOp} holds. 
We now characterize the space $\mathcal{X}:=\left(\Ell^p(\Omega,X)\right)_{-1}(\M)$. Assume 
that the Banach space $X$ we are working with is separable, i.e., there exists a countable dense set in $X$. Denote the extrapolation spaces corresponding to the fiber operator $(M(s),\dom(M(s))$ by $\left(X_{-1,s},\left\|\cdot\right\|_{-1,s}\right)$, $s\in\Omega$. The following result prepares for the extrapolation procedure. .
\begin{lemma}\label{lem:BFSExtrap}
Suppose $X$ is a separable Banach space. If $0\in\rho(M(s))$ for almost every $s\in\Omega$ and $s\mapsto M(s)^{-1}x$ is measurable for each $x\in X$, then the family $\left(X_{-1,s},\left\|\cdot\right\|_{-1,s}\right)_{s\in\Omega}$ is a measurable Banach fiber bundle.
\end{lemma}

\begin{proof}
We make use of the separability of $X$ and take a dense countable subset of $X$ and make a $\QQ+i\QQ$ vector space $\mathcal{B}$ out of it. Observe that $\mathcal{B}$ is still countable, i.e, $\B:=\left\{b_k:\ k\in\NN\right\}\subseteq X$. We define a family of seminorms $\left\{\opnorm{\cdot}_s:\ s\in\Omega\right\}$ on $X$ by
\[
\opnorm{x}_{s}:=\left\|M(s)^{-1}x\right\|,\quad x\in X,\ s\in\Omega.
\]
Then $\opnorm{\cdot}_s$ is actually a norm on $X$. By the assumption $s\mapsto M(s)^{-1}x$ is measurable for each $x\in X$ and hence so is the map $s\mapsto\left\|b_k\right\|_s$ for each $k\in\NN$. Since $N_s=\left\{0\right\}$ for each $s\in\Omega$ we obtain $\B/N_s=\B$. Finally, the completion of $\B$ with respect to $\opnorm{\cdot}_s$ is just the space $X_{-1,s}$, $s\in\Omega$. By Definition \ref{def:mBfs} we therefore obtain that $\left(X_{-1,s},\left\|\cdot\right\|_{-1,s}\right)_{s\in\Omega}$ is a measurable Banach fiber bundle.
\end{proof}

Since we know that $\left(X_{-1,s},\left\|\cdot\right\|_{-1,s}\right)_{s\in\Omega}$ is a measurable Banach fiber bundle, we can consider the space of fiber $p$-integrable functions over this set of Banach spaces. In what follows we relate this space to the extrapolation space of $\Ell^p(\Omega,(X_s)_{s\in\Omega})$ with respect to the 
 operator-valued multiplication operator $(\M,\dom(\M))$.

\begin{theorem}\label{cor:FibLp}
Let $1\leq p<\infty$ and consider the unbounded multiplication operator $(\M,\dom(\M))$ on $\Ell^p(\Omega,X)$, induced by the family of unbounded operators $(M(s),\dom(M(s)))_{s\in\Omega}$ on $X$. Let $(M(s),\dom(M(s)))$ be a semigroup generator for $\mu$-almost every $s\in\Omega$. Suppose that $0\in\rho(M(s))$ for $\mu$-almost every $s\in\Omega$ and that $s\mapsto M(s)b$ and $s\mapsto M(s)^{-1}b$ are measurable for each $b\in\B$. Then
\[
\left[\Ell^p(\Omega,X)\right]_{-1}(\M)=\Ell^p(\Omega,(X_{-1,s})_{s\in\Omega}).
\]
\end{theorem}

\begin{proof}
Let $f\in\left[\Ell^p(\Omega,X)\right]_{-1}(\M)$ and find $g\in\Ell^p(\Omega,X)$ such that $f=\mathcal{M}_{-1}g$, where $\M_{-1}:\Ell^p(\Omega,X)\to\Ell^p(\Omega,X)_{-1}(\M)$. Since $g$ is measurable, we can find a sequence $(g_n)_{n\in\NN}$ of simple functions approximating $g$ pointwise, i.e.,
\[
g_n:=\sum_{i=1}^{m_n}{x_{n,i}\textbf{1}_{\Omega_{n,i}}}\ \text{and}\ g_n\to g\ \mu\text{-almost everywhere.}
\]
Without loss of generality, we assume that $x_{n,i} \in \B$ for every $i=1,2,\ldots m_n$ and $n\in\mathbb{N}$.
In order to show that $f\in\Ell^p(\Omega,(X_{-1,s})_{s\in\Omega})$ define $f_n:=\M_{-1}g_n$. Then
\[
f_n(s):=(\M_{-1}g_n)(s)=\sum_{i=1}^{m_n}{M_{-1}(s)x_{n,i}\textbf{1}_{\Omega_{n,i}}(s)},
\]
where we use Lemma \ref{lem:ExtMultSemiAssumpGen} as well as the fact that $M_{-1}(s)x_{n,i}\in X_{-1,s}$, $s\in\Omega$. Since the function $\Omega\rightarrow X; s\mapsto M(s)b$ is measurable for each $n\in\B$, it is easy to see that $f_n$ is fiber-measurable for each $n\in\mathbb{N}$.

 Furthemore,
\[
\left\|f\right\|^p_{\Ell^p\left(\Omega,(X_{-1,s})_{s\in\Omega}\right)}=\int_\Omega{\left\|f(s)\right\|_{-1,s}^p\ \dd{s}}=\int_{\Omega}{\left\|g(s)\right\|^p\ \dd{s}}=\left\|g\right\|^p_{\Ell^p(\Omega,X)}<\infty,
\]
showing that $f$ is a fiber $p$-integrable function, i.e., $f\in\Ell^p(\Omega,(X_{-1,s})_{s\in\Omega})$. 

\medskip
For the converse inclusion, suppose that $f\in\Ell^p(\Omega,(X_{-1,s})_{s\in\Omega})$. We have to show that there exists $g\in\Ell^p(\Omega,X)$ such that $f=\M_{-1}g$. Since $f$ is fiber measurable, there exists a sequence $(f_j)_{j\in\NN}$ of simple functions $f_j:\Omega\to\bigcup_{s\in\Omega}{X_{-1,s}}$ with $f(s)\in X_{-1,s}$ for $\mu$-almost every $s\in\Omega$ and
\[
f_j=\sum_{i=1}^{n_j}{b_{k_i}\textbf{1}_{\Omega_i}},
\]
where $b_{k_i}\in\mathcal{B}$, $\Omega_i\in\Sigma$, $\Omega_i\cap\Omega_j=\varnothing$, $i\neq j$, for $1\leq i\leq n_j$, and 
\begin{align}\label{eqn:CauchyFiber}
\left\|f(s)-f_j(s)\right\|_{-1,s}\to0,
\end{align}
for $j\to\infty$ and $\mu$-almost every $s\in\Omega$. By the assumption that $0\in\rho(M(s))$ for $\mu$-almost every $s\in\Omega$ we conclude by Proposition \ref{prop:ResolventMultFiber} that $0\in\rho(\M)$. So we define
\[
g_j:=(\M^{-1}_{-1} f_j)(\cdot)=\sum_{i=1}^{n_j}{\left(M^{-1}_{-1}(\cdot)b_{k_i}\right)\textbf{1}_{\Omega_i}(\cdot)}.
\]
We observe that $M_{-1}^{-1}(s)b_{k_i}\in X$ for $\mu$-almost every $s\in\Omega$ and hence $g_j$ is a simple function. By \eqref{eqn:CauchyFiber} we conclude that $(g_j(s))_{j\in\NN}$ is a Cauchy sequence with respect to $\left\|\cdot\right\|_s$ for $\mu$-almost every $s\in\Omega$ and hence convergent. This yields a measurable function $g:\Omega\to X$ by taking the pointwise limit, i.e., 
\[
g(s):=\lim_{j\to\infty}{g_j(s)},\quad s\in\Omega.
\]
By the continuity of $M_{-1}(s)$, $s\in\Omega$, on $X$ and the fact that $M_{-1}(s)M_{-1}^{-1}(s)=\Id$ for $\mu$-almost every $s\in\Omega$, we directly obtain that $\M_{-1}g=f$. Moreover,
\[
\left\|g\right\|_{\Ell^p(\Omega,X)}^p=\int_{\Omega}{\left\|g(s)\right\|^p\ \dd{s}}=\int_{\Omega}{\left\|f(s)\right\|_{-1,s}^p\ \dd{s}}=\left\|f\right\|^p_{\Ell^p(\Omega,(X_{-1,s})_{s\in\Omega}}<\infty,
\]
and therefore $g\in\Ell^p(\Omega,X)$.
\end{proof}

\CB{As a direct consequence we recover the following result, which has previously been used in in the continuous setting in \cite[Sect.~4]{NagelIdent}.

\begin{corollary}
Let $(A,\dom(A))$ be a generator of a $C_0$-semigroup on a Banach space $X$ and denote ist first extrapolation space by $X_{-1}^A$. Define the operator $(\M,\dom(\M))$ on $\Ell^p(\Omega,X)$ by
\[
(\M f)(s):=Af(s),\quad f\in\dom(\M):=\Ell^p(\Omega,\dom(A)),\ s\in\Omega.
\]
Then $\left(\Ell^p(\Omega,X)\right)_{-1}(\M)=\Ell^p(\Omega,X_{-1}^A)$.
\end{corollary}
}

\section*{Acknowledgement}
We would like to thank B\'{a}lint Farkas for a discussion in which the idea for the collaboration which resulted in this paper was initiated. Furthermore, we are indepted to Rainer Nagel for helpful feedback.

\end{document}